\documentclass[11pt,reqno]{amsart}

\usepackage{color}

\usepackage{amsfonts,amscd}
\usepackage{amssymb}

\usepackage[english]{babel}

\usepackage{tikz}

\theoremstyle{plain}
\newtheorem{theorem}                 {Theorem}      [section]

\newtheorem{lemma}        [theorem]  {Lemma}
\newtheorem{proposition}  [theorem]  {Proposition}

\theoremstyle{definition}

\newtheorem{remark}       [theorem]  {Remark}

\numberwithin{equation}{section}
\numberwithin{table}{section}

\def \cn{{\mathbb C}}
\def \hn{{\mathbb H}}

\def \rn{{\mathbb R}}

\def \B{\mathcal B}

\def\nab#1#2{\hbox{$\nabla$\kern -.3em\lower 1.0 ex
    \hbox{$#1$}\kern -.1 em {$#2$}}}

\def \Re{\mathfrak R\mathfrak e}

\def \g{\mathfrak{g}}

\def \k{\mathfrak{k}}

\def \m{\mathfrak{m}}

\def \un{\mathfrak{u}}

\DeclareMathOperator{\Ad}{Ad}

\def \GLR#1{\text{\bf GL}_{#1}(\rn)}

\def \GLC#1{\text{\bf GL}_{#1}(\cn)}
\def \glc#1{\mathfrak{gl}_{#1}(\cn)}
\def \GLH#1{\text{\bf GL}_{#1}(\hn)}

\def \SLR#1{\text{\bf SL}_{#1}(\rn)}
\def \SL2{\widetilde{\text{\bf SL}}_{2}(\rn)}

\def \SLC#1{\text{\bf SL}_{#1}(\cn)}

\def \O#1{\text{\bf O}(#1)}
\def \SO#1{\text{\bf SO}(#1)}

\def \SOs#1{\text{\bf SO}^*(#1)}

\def \U#1{\text{\bf U}(#1)}

\def \SU#1{\text{\bf SU}(#1)}
\def \SUs#1{\text{\bf SU}^*(#1)}

\def \Sp#1{\text{\bf Sp}(#1)}
\def \sp#1{\mathfrak{sp}(#1)}

 \DeclareMathOperator{\grad}{grad}

\DeclareMathOperator{\trace}{trace}

\numberwithin{equation}{section}
\allowdisplaybreaks

\begin{document}
	
%%%%%%%%%%%%%%%%%%%%%%%%%%%%%%%%%%%%%%%%%%%%%%%%%%%%%%%%

\title[$p$-harmonic functions on Riemannian symmetric spaces]{Explicit proper $p$-harmonic functions\\ on the Riemannian  symmetric spaces\\ $\SU n/\SO n$, $\Sp n/\U n$, $\SO {2n}/\U n$, $\SU{2n}/\Sp n$}

\author{Sigmundur Gudmundsson}
\address{Mathematics, Faculty of Science\\
University of Lund\\
Box 118, Lund 221 00\\
Sweden}
\email{Sigmundur.Gudmundsson@math.lu.se}

\author{Anna Siffert}
\address{Mathematisches Institut\\
	Einsteinstr. 62\\
	48149 M\" unster\\
	Germany}
\email{ASiffert@uni-muenster.de}

\author{Marko Sobak}
\address{Mathematisches Institut\\
	Einsteinstr. 62\\
	48149 M\" unster\\
	Germany}
\email{MSobak@uni-muenster.de}

\begin{abstract}
In this work we construct explicit complex-valued $p$-harmonic functions on the compact Riemannian symmetric spaces $\SU n/\SO n$, $\Sp n/\U n$, $\SO {2n}/\U n$, $\SU{2n}/\Sp n$.  We also describe how the same can be manufactured on their non-compact symmetric dual spaces.
\end{abstract}

%\dedicatory{version 2.003 - \today}



\maketitle

%%%%%%%%%%%%%%%%%%%%%%%%%%%%%%%%%%%%%%%%%%%%%%%%%%%%%%%%

\section{Introduction}\label{section-introduction}

The main objects for this study are the \textit{$p$-harmonic functions} $\phi: (M,g) \to \cn$ on Riemannian manifolds. These are solutions to the $p$-harmonic equation
\begin{equation*}
\tau^p(\phi) = 0,
\end{equation*}
where $\tau$ denotes the Laplace-Beltrami operator on $M$ and $\tau^p$ means $\tau$ applied $p$-times. 
Since each $p$-harmonic function is trivially $r$-harmonic for any $r \geq p$, we are interested the lowest such value.
\smallskip

The study of $p$-harmonic functions on Riemannian manifolds has invoked the interest of mathematicians and physicists for nearly two centuries. Applications within physics can for example be found in continuum mechanics, elasticity theory, as well as two-dimensional hydrodynamics problems involving Stokes flows of incompressible Newtonian fluids. 
\smallskip

Until just a few years ago, with only very few exceptions, the domains of all known explicit proper $p$-harmonic functions
have been either surfaces or open subsets of flat Euclidean space. 
The recent development has changed this situation and can be traced at the regularly updated online bibliography \cite{Gud-p-bib}, maintained by the first author. 
\smallskip

In their work \cite{Gud-Sob-1}, Gudmundsson and Sobak develop a general method for constructing proper $p$-harmonic functions using \textit{eigenfunctions} on the domain manifold $M$. These are  functions $\phi : (M,g) \to \cn$ such that there exist constants $\lambda,\mu \in \cn$, not both zero, with 
\begin{equation*}
\tau(\phi) = \lambda\cdot\phi\ \ \text{and}\ \  \kappa(\phi,\phi) = \mu\cdot\phi^2,
\end{equation*}
where $\kappa$ denotes the complex bilinear {\it conformality operator} defined by
\begin{equation*}
\kappa(\phi,\psi) = g(\grad \phi, \grad \psi).
\end{equation*}
More explicitly, they show that if $\phi$ is an eigenfunction on $M$, then the function
\begin{equation*}
\Phi_p:x\mapsto
\begin{cases}
	c_1\cdot\log(\phi(x))^{p-1}, 							& \text{if }\; \mu = 0, \; \lambda \not= 0\\[0.2cm]	c_1\cdot\log(\phi(x))^{2p-1}+ c_{2}\cdot\log(\phi(x))^{2p-2}, 								& \text{if }\; \mu \not= 0, \; \lambda = \mu\\[0.2cm]
	c_{1}\cdot\phi(x)^{1-\frac\lambda{\mu}}\log(\phi(x))^{p-1} + c_{2}\cdot\log(\phi(x))^{p-1},	& \text{if }\; \mu \not= 0, \; \lambda \not= \mu
	\end{cases}
\end{equation*}
is proper $p$-harmonic on an appropriate open subdomain of $M$.
Even though this method of construction seems simple at first glance, it certainly comes at a cost - it relies on the existence of eigenfunctions, which are not easy to find in general.
In fact, it is usually not difficult to spot functions which are eigenfunctions with respect to the Laplace-Beltrami operator due to its linearity, but finding functions which are also eigenfunctions with respect to the conformality operator can be quite a daunting task.

\medskip

\begin{table}[t]
\small
	
\renewcommand{\arraystretch}{2}
	
\makebox[\textwidth][c]
{
\begin{tabular}{llccr}
\hline
$U/K$ 	& Eigenfunction  			& $\lambda$ 				& $\mu$ 			& Conditions
\\ 
\hline
$\SU n / \SO n$ & $\trace(z^t A z)$ 		& $-\frac{2(n^2+n-2)}{n}$ 		& $-\frac{4(n-1)}{n}$	& $A = aa^t$ for $a \in \cn^{n}$
\\
$\Sp n / \U n$ & $\trace(q^t A q)$			& $-2(n+1)$						& $-2$					& $A = aa^t$ for $a \in \cn^{2n}$
\\
$\SO{2n} / \U n$ & $\trace(x^t A xJ_n)$			& $-2(n-1)$						& $-1$				& $A = ab^t$ for $a,b \in V \subset \cn^{2n}$ isotropic\\
$\SU{2n} / \Sp n$ & $\trace(z^t A zJ_n)$			& $-\frac{2(2n^2-n-1)}{n}$						& $-\frac{2(n-1)}{n}$	& $A = ab^t$ for $a,b \in \cn^{2n}$
\\
\hline
\end{tabular}
}
\bigskip
\caption{Eigenfunctions on the compact symmetric spaces.}
\label{table-eigenfunctions}
\end{table}

In this paper, we consider the compact symmetric spaces
$$\SU n/\SO n,\ \Sp n/\U n,\ \SO {2n}/\U n,\ \SU{2n}/\Sp n,$$ 
as well as their non-compact duals
$$\SLR n / \SO n,\ \Sp{n,\rn} / \U n,\ \SOs{2n} / \U n,\ \SUs{2n} / \Sp n.$$
Eigenfunctions on the dividend group $G$ are already known for all of the above mentioned cases, see e.g.\ \cite{Gud-Sak-1} or \cite{Gud-Sob-1}. 
However, the issue with the known eigenfunctions is that they are not invariant under the action of the divisor group $K$, and hence do not induce eigenfunctions on the corresponding quotient space $G/K$. The $K$-invariance condition adds an additional degree of difficulty to the problem of finding eigenfunctions.
Our aim is therefore to construct new examples of eigenfunctions which are indeed $K$-invariant, hence induce proper $p$-harmonic functions on the quotient spaces $G/K$ via the method presented above. We summarize our results for the compact spaces in Table \ref{table-eigenfunctions}. Note that these constructions are sufficient, since the general duality principle developed in \cite{Gud-Sve-1} automatically yields examples of eigenfunctions on the non-compact duals as well.

\medskip

\noindent\textbf{Organisation:}
In Section \ref{section-duality} we give a general description of $p$-harmonic functions in the setting of Riemannian symmetric spaces and state a useful duality principle for those.  In Section \ref{section-relevant-groups} we discuss the necessary details for the Lie groups, relevant to this study.  In Sections \ref{section-SUn-SOn} and \ref{section-SO2n-Un}, we prove our main results, which are displayed in Table \ref{table-eigenfunctions}.

%%%%%%%%%%%%%%%%%%%%%%%%%%%%%%%%%%%%%%%%%%%%%%%%%%%%%%%%

\section{$p$-Harmonic Functions on Symmetric Spaces}
\label{section-duality}

Let $(M,g)$ be a Riemannian symmetric space of non-compact type, $G$ be the connected component of the isometry group containing the neutral element and $K$ be the maximal compact subgroup.  For the Lie algebra $\g$ of $G$, we have the  Cartan decomposition $\g=\k\oplus\m$, where $\k$ is the Lie algebra of $K$ and $\m$ its orthogonal complement with respect to the $\Ad (K)$-invariant metric on $G$ induced by its Killing form. Let $\phi:U\to\cn$ be a locally defined real analytic function and $\hat\phi=\phi\circ\pi:G\to\cn$ be the $K$-invariant composition of $\phi$ with the natural projection $\pi:G\to G/K$. Then the Laplace-Beltrami operator $\hat\tau$ and the conformality operator $\hat\kappa$ on $G$ satisfy
$$\hat\tau(\hat\phi)=\sum_{Z\in\m}\big( Z^2(\hat\phi)-\nab ZZ(\hat\phi)\bigr),\ \ \hat\kappa (\hat\phi,\hat \phi)=\sum_{Z\in\m}Z(\hat f)^2.$$
We observe that
\begin{equation*}
	\tau(\phi)=\tau(\hat\phi \circ \pi)
	=d\hat\phi(\tau(\pi)) \circ \pi + \trace\nabla d\hat\phi(d\pi, d\pi) \circ \pi
	=\tau(\hat\phi)\circ \pi,
\end{equation*}
where the equalities follow from the facts that $\pi$ is harmonic and the restriction of $d\pi$ to the orthogonal complement of its kernel is an isometry onto the tangent space of $G/K$ at the corresponding point.
The latter fact also implies that
\begin{equation*}
	\kappa(\phi,\phi) = \kappa(\hat\phi,\hat\phi)\circ\pi.
\end{equation*}
This shows that a $K$-invariant eigenfunction $\hat\phi$ on $G$ induces an eigenfunction $\phi$ on the quotient space $G/K$ with exactly the same eigenvalues.
\smallskip

We now extend the real analytic $\hat\phi$ to a holomorphic function $\hat \phi^*:W^*\to\cn$ locally defined on the complexification $G^\cn$ of $G$.  The complex Lie group $G^\cn$ contains the compact subgroup $U$ with Lie algebra $\un=\k\oplus i\,\m$.  Let $\hat\phi^*:W^*\cap U\to\cn$ be the $K$-invariant restriction of $\hat\phi^*$ to $W^*\cap U$.  Then this induces a function $\phi^*:\pi^*(W^*\cap U)\to\cn$ defined locally  on the symmetric space $U/K$ which is the compact companion of $G/K$.  Here $\pi^*:U\to U/K$ is the corresponding natural projection. Then the Laplace-Beltrami operator $\hat\tau^*$ and the conformality operator $\hat\kappa^*$ on $U$ satisfy
$$\hat\tau(\hat \phi^*)=\sum_{Z\in\, i\,\m}\big( Z^2(\hat \phi^*)-\nab ZZ(\hat f^*)\bigr)=-\sum_{Z\in\,\m}\big( Z^2(\hat f^*)-\nab ZZ(\hat f^*)\bigr)$$
and
$$\hat\kappa^* (\hat f^*,\hat f^*)=\sum_{Z\in\, i\,\m}Z(\hat f^*)^2=-\sum_{Z\in\,\m}Z(\hat f^*)^2.$$

For $p$-harmonic function on symmetric spaces we have the following duality principle.  This was first introduced for harmonic morphisms in \cite{Gud-Sve-1} and then developed further for our situation in \cite{Gud-Mon-Rat-1}.

\begin{theorem}\label{theorem-duality}
	A complex-valued function $f:W\to\cn$ is proper p-harmonic if and only if its dual $f^*:W^*\to\cn$ is proper p-harmonic.
\end{theorem}

\begin{proof}
	For a proof of this result we recommend Section 8 of \cite{Gud-Mon-Rat-1}.
\end{proof}

\begin{remark}
	It should be noted that if $\phi:G/K\to\cn$ is an eigenfunction on the non-compact symmetric space $G/K$, such that 
	$$\tau(\phi)=\lambda\cdot\phi\ \ \text{and}\ \ \kappa(\phi,\phi)=\mu\cdot\phi^2,$$ 
	then its dual function $\phi^*:U/K\to\cn$, on the compact $U/K$, fulfills 
	$$\tau(\phi)=-\lambda\cdot\phi\ \ \text{and}\ \ \kappa(\phi,\phi)=-\mu\cdot\phi^2.$$ 
\end{remark}

%%%%%%%%%%%%%%%%%%%%%%%%%%%%%%%%%%%%%%%%%%%%%%%%%%%%%%

%%%%%%%%%%%%%%%%%%%%%%%%%%%%%%%%%%%%%%%%%%%%%%%%%%%%%%%%

\section{The Relevant Lie Groups}\label{section-relevant-groups}

Let us first recall the definitions of the classical Lie groups relevant for this study.  We refer the interested reader to the standard work \cite{Hel} which goes far beyond our presentation.

\smallskip
We denote the $n \times n$ identity matrix by $I_n$ and the standard complex structure on $\rn^{2n}$ by
\begin{equation*}
J_n = \begin{bmatrix}
0 & I_n\\
-I_n & 0
\end{bmatrix}.
\end{equation*}
Furthermore, for $1 \leq r,s \leq n$, we denote by $E_{rs}$ the $n\times n$ matrix with components $(E_{rs})_{\alpha\beta} = \delta_{r\alpha} \delta_{s\beta}$,
and we further define
\begin{equation*}
 X_{rs} = \tfrac{1}{\sqrt 2} (E_{rs} + E_{sr}), \quad Y_{rs} = \tfrac{1}{\sqrt 2} (E_{rs} - E_{sr}), \quad D_r = E_{rr}.
 \end{equation*} 

If $\mathbb F$ is either the field of the real numbers $\rn$ or that of the complex $\cn$, then the general linear group $\mathbf{GL}_n(\mathbb F)$ is defined as the set of all invertible matrices $x \in \mathbb F^{n\times n}$ and the special linear group $\mathbf{SL}_n(\mathbb F)$ is the subgroup of matrices in $\mathbf{GL}_n(\mathbb F)$ whose determinant is $1$.
The compact orthogonal and unitary groups are given by
\begin{eqnarray*}
\O n &=& \{ x \in \GLR n \,\mid\, x\cdot x^t = I_n \},\\
\U n &=& \{ z \in \GLC n \,\mid\, z\cdot z^* = I_n \},
\end{eqnarray*}
and their special counterparts are obtained by intersecting with the corresponding special linear groups, so that
\begin{eqnarray*}
\SO n &=& \O n \cap \SLR n,\\
\SU n &=& \U n \cap \SLC n.
\end{eqnarray*}
The quaternionic unitary group $\Sp n$ is defined as the intersection of the standard embedding
\begin{equation*}
\GLH n \ni z+jw \mapsto q = \begin{bmatrix}
z & w\\
-\bar{w} & \bar{z}
\end{bmatrix} \in \GLC{2n}
\end{equation*}
and the unitary group $\U{2n}$.
Some further interpretations are needed in order to understand the definitions of the relevant quotient spaces:

\begin{enumerate}
\item[(1)] $\SO n$ is a subgroup of $\SU n$, so the quotient space $\SU n / \SO n$ is naturally defined.
\vskip0.2cm
\item[(2)] $\U n$ can be embedded into $\Sp n$, via the map
\begin{equation}\label{eq-embed-un}
x+iy \mapsto \begin{bmatrix}
x & y\\
-y & x
\end{bmatrix},
\end{equation}
as shown by a simple calculation. Hence, we can view $\U n$ as a subgroup of $\Sp n$ and thus consider the quotient space $\Sp n / \U n$.

\vskip0.2cm

\item[(3)] $\U n$ can also be embedded into $\SO{2n}$ via the mapping (\ref{eq-embed-un}).
Indeed, the identity
\begin{equation*}
|\det(x+iy)|^2 = \det \begin{bmatrix}
x & y\\
-y & x
\end{bmatrix},
\end{equation*}
together with a simple calculation, shows that the image is contained in $\SO{2n}$.  Via this embedding, $\U n$ becomes a subgroup of $\SO{2n}$, allowing us to consider the quotient space $\SO{2n} / \U n$.

\vskip0.2cm

\item[(4)] $\Sp n$ is by definition a subgroup of $\U{2n}$, and it can be shown that the elements of $\Sp n$ have unit determinant, so the quotient space $\SU{2n} / \Sp n$ is well-defined. 
\end{enumerate}

Using the standard left-invariant Riemannian metric on $\GLC n$, induced by the inner product
\begin{equation*}
g(Z,W) = \Re\trace (ZW^*)
\end{equation*}
on the Lie algebra $\glc n$, we obtain the induced Riemannian metric on every Lie subgroup $G$ of $\GLC n$.
\smallskip

Let us now consider the standard complex linear representations of the compact Lie groups $\SO n$, $\SU n$ and $\Sp n$.  According to the Peter-Weyl theorem their matrix coefficients are all eigenfunctions of the Laplace-Beltrami operator with the same eigenvalue in each case.  For obvious reasons we are also interested in how they behave with respect to the corresponding comformality operator.  For the special orthogonal group $\SO n$ these matrix coefficients are the coordinate functions $x_{j\alpha}:\SO n\to\rn$ satisfying
\begin{equation*}
x = \begin{bmatrix}
x_{11} & \cdots & x_{1n}\\
\vdots & \ddots & \vdots\\
x_{n1} & \cdots & x_{nn}
\end{bmatrix}
\mapsto
x_{j\alpha}.
\end{equation*}
For the standard irreducible representations of $\SO n$, $\SU n$ and $\Sp n$ we have the following:
\begin{enumerate}
\item[(i)] For $1 \leq j,\alpha,k,\beta \leq n$, the matrix coefficients  $x_{j\alpha} : \SO n \to \rn$ satisfy
\begin{eqnarray}\label{eq-coordinate-functions-SOn}
\tau(x_{j\alpha}) &=& -\frac{n-1}{2}\cdot x_{j\alpha}, \nonumber\\
\kappa(x_{j\alpha}, x_{k\beta}) &=& -\frac{1}{2}\cdot (x_{j\beta}x_{k\alpha} - \delta_{jk}\delta_{\alpha\beta}),
\end{eqnarray}
Here we refer to Lemma 4.1 of \cite{Gud-Sak-1}.
\vskip0.2cm

\item[(ii)] For $1 \leq j,\alpha,k,\beta \leq n$, the matrix coefficients  $z_{j\alpha} : \SU n \to \cn$ satisfy
\begin{eqnarray}\label{eq-coordinate-functions-SUn}
\tau(z_{j\alpha}) &=& -\frac{n^2-1}{n}\cdot z_{j\alpha}, \nonumber\\
\kappa(z_{j\alpha}, z_{k\beta}) &=& - z_{j\beta}z_{k\alpha} + \frac{1}{n} \cdot z_{j\alpha} \, z_{k\beta},
\end{eqnarray}
This can be proven by utilising Lemma 5.1 of \cite{Gud-Sak-1} and a simple idea explained at the end of Section 4 in \cite{Gud-Sob-1}.
\vskip0.2cm

\item[(iii)] For $1 \leq j,\alpha,k,\beta \leq 2n$, the matrix coefficients $q_{j\alpha}: \Sp n \to \cn$ satisfy
\begin{eqnarray}\label{eq-coordinate-functions-Spn}
\tau(q_{j\alpha}) &=& -\frac{2n+1}{2}\cdot q_{j\alpha}, \nonumber\\ 
\kappa(q_{j\alpha}, q_{k\beta}) &=& -\frac{1}{2}\cdot q_{j\beta}q_{k\alpha} + \frac12 (J_n)_{jk} (J_n)_{\alpha\beta}.
\end{eqnarray}
These two identities form an improved version of Lemma 6.1 in \cite{Gud-Mon-Rat-1}.
We present a proof of these formulae in Appendix \ref{appendix-coordinate-spn}.
\end{enumerate}

%%%%%%%%%%%%%%%%%%%%%%%%%%%%%%%%%%%%%%%%%%%%%%%%%%%%%%%%

\section{Eigenfunctions on $\SU n/\SO n$ and $\Sp n / \U n$}
\label{section-SUn-SOn}

The aim of this section is to prove the statements presented in Table \ref{table-eigenfunctions} concerning the compact symmetric spaces $\SU n/\SO n$ and $\Sp n / \U n$. 

\medskip

First, consider the map $\Phi : \SU n \to \SU n$ defined by
\begin{equation*}
\Phi(z) = z \cdot z^t.
\end{equation*}
Note that this map is $\SO n$-invariant, so that for any $f:\SU n\to\cn$, the composition $\phi=f\circ\Phi$ induces a function on the compact symmetric quotient space $\SU n / \SO n$.
As already indicated in Table \ref{table-eigenfunctions}, we will consider the case when $f$ is linear i.e. when
\begin{equation*}
f(z) = \sum_{j\alpha} A_{j\alpha}z_{j\alpha} = \trace(Az)
\end{equation*}
for some symmetric matrix $A\in\cn^{n\times n}$. Note that the symmetry condition on $A$ can be assumed, without loss of generality, since $\Phi(z)$ itself is symmetric and the product of a skew-symmetric and a symmetric matrix is traceless.

\begin{proposition}\label{proposition-SUn/SOn}
Let the complex symmetric matrix $A$ be given by $A = aa^t$ for some non-zero element $a\in\cn^n$. Further consider the function
$\phi:\SU n\to\cn$ with
\begin{equation*}
\phi(z) = \trace(A\Phi(z)) = \trace(z^tAz) = \sum_{j,\alpha}a_j\,a_\alpha\, \Phi_{j\alpha}(z).
\end{equation*}
Then $\phi$ is an $\SO n$-invariant eigenfunction on $\SU n$ satisfying 
\begin{equation*}
\tau(\phi) = -\frac{2(n^2+n-2)}{n} \cdot \phi
\quad\text{and}\quad
\kappa(\phi,\phi) = -\frac{4(n-1)}{n} \cdot \phi^2.
\end{equation*}
\end{proposition}

\begin{proof} 
Observe that we can write
$$\Phi_{j\alpha}(z)=\sum_{r=1}^{n}z_{jr}z_{\alpha r}.$$
Then we see from equations (\ref{eq-coordinate-functions-SUn}) that 
\begin{eqnarray*}
\tau(\Phi_{j\alpha})
&=&\sum_{r=1}^{n}\tau(z_{jr}z_{\alpha r})\\
&=&\sum_{r=1}^{n}(\tau(z_{jr})z_{\alpha r}+2\cdot \kappa(z_{jr},z_{\alpha r})+z_{jr}\tau(z_{\alpha r}))\\
&=&-\sum_{r=1}^{n}\left(\frac{n^2-1}{n} + 2 \left(1 - \frac{1}{n}\right) + \frac{n^2-1}{n}\right)z_{jr}z_{\alpha r}\\ 
&=&-\frac{2(n^2+n-2)}{n}\cdot\Phi_{j\alpha}.
\end{eqnarray*}
The formula for $\tau(\phi)$ thus follows immediately since $\tau$ is linear.

As for the conformality operator, we have
\begin{eqnarray*}
\kappa(\Phi_{j\alpha},\Phi_{k\beta})
&=&\sum_{r,s=1}^{n}\kappa(z_{jr}z_{\alpha r},z_{ks}z_{\beta s})\\
&=&\sum_{r,s=1}^{n}
\Bigl(z_{\alpha r}z_{\beta s}\kappa(z_{jr},z_{ks}) + z_{\alpha r}z_{ks}\kappa(z_{jr},z_{\beta s})\\
& &\qquad\quad + z_{jr}z_{\beta s}\kappa(z_{\alpha r},z_{ks}) + z_{jr}z_{ks}\kappa(z_{\alpha r},z_{\beta s})\Bigr)\\
&=& \sum_{r,s=1}^{n} \left(
-z_{\alpha r}z_{\beta s}z_{js}z_{kr} + \frac{1}{n}\cdot z_{\alpha r}z_{\beta s}z_{jr}z_{ks} \right.\\
& & \qquad\quad - z_{\alpha r}z_{ks}z_{js}z_{\beta r} + \frac{1}{n}\cdot z_{\alpha r}z_{ks}z_{jr} z_{\beta s}\\
& & \qquad\quad - z_{jr}z_{\beta s}z_{\alpha s}z_{kr} + \frac{1}{n} \cdot z_{jr}z_{\beta s}z_{\alpha r}z_{ks} \\
& & \left. \qquad\quad - z_{jr}z_{ks}z_{\alpha s}z_{\beta r} + \frac{1}{n} \cdot z_{jr}z_{ks}z_{\alpha s}z_{\beta r} \right)\\
&=&-2\cdot \Phi_{j\beta}\Phi_{k\alpha} - 2\cdot \Phi_{jk}\Phi_{\alpha\beta} + \frac{4}{n} \cdot \Phi_{j\alpha}\Phi_{k\beta},
\end{eqnarray*}
and thus by bilinearity
\begin{eqnarray*}
\kappa(\phi,\phi)
&=&
\sum_{j,\alpha, k,\beta}
a_j\,a_\alpha \,a_k\,a_\beta\cdot \kappa(\Phi_{j\alpha},\Phi_{k\beta})\\
&=&
-2 \left( \sum_{j,\beta} a_j \,a_\beta\cdot\Phi_{j\beta} \right)\left( \sum_{k,\alpha} a_k \,a_\alpha\cdot\Phi_{k\alpha} \right)\\
&&
-2 \left( \sum_{j,k} a_j \,a_k\cdot\Phi_{jk} \right)\left( \sum_{\alpha,\beta} a_\alpha \,a_\beta\cdot\Phi_{\alpha\beta} \right)\\
&&
+\frac{4}{n} \left( \sum_{j,\alpha} a_j \,a_\alpha\cdot\Phi_{j\alpha} \right)\left( \sum_{k,\beta} a_k \,a_\beta\cdot\Phi_{k\beta} \right)\\
&=&
-\frac{4(n-1)}{n} \cdot \phi^2,
\end{eqnarray*}
as claimed.
\end{proof}

\medskip

As for $\Sp n / \U n$, the construction is quite similar but the proof requires a minor modification.
Here, we instead consider the $\U n$-invariant map $\Phi : \Sp n \to \Sp n$ given by
\begin{equation*}
\Phi(q) = q \cdot q^t,
\end{equation*}
and as before we take the trace of its product with a symmetric matrix.

\begin{proposition}
Let the complex symmetric matrix $A$ be given by $A = aa^t$ for some non-zero element $a\in\cn^{2n}$. Further consider the function
$\phi:\Sp n\to\cn$ with
\begin{equation*}
\phi(q) = \trace(A\Phi(q)) = \trace(q^tAq) = \sum_{j,\alpha}a_j\,a_\alpha\, \Phi_{j\alpha}(q).
\end{equation*}
Then $\phi$ is an $\U n$-invariant eigenfunction on $\Sp n$ satisfying 
\begin{equation*}
\tau(\phi) = -2(n+1) \cdot \phi
\quad\text{and}\quad
\kappa(\phi,\phi) = -2 \cdot \phi^2.
\end{equation*}
\end{proposition}

\begin{proof}
Using similar techniques as in the proof of Proposition \ref{proposition-SUn/SOn} together with the fact that $q \cdot J_n \cdot q^t = J_n$ for all $q \in \Sp n$, one first shows that
\begin{eqnarray*}
\tau(\Phi_{j\alpha})
&=&-2(n+1)\cdot\Phi_{j\alpha},\\
\kappa(\Phi_{j\alpha},\Phi_{k\beta})
&=& -(\Phi_{k\alpha}\Phi_{j\beta} + \Phi_{jk}\Phi_{\alpha\beta})  + (J_n)_{\alpha k} (J_n)_{j\beta} + (J_n)_{jk} (J_n)_{\alpha\beta},
\end{eqnarray*}
from which it easily follows that
\begin{eqnarray*}
\tau(\phi) &=& -2(n+1)\cdot\phi\\
\kappa(\phi,\phi)
&=& -2\cdot \phi^2 + 2 \cdot (\trace(AJ_n))^2 = -2\cdot \phi^2,
\end{eqnarray*}
where the final equality follows since the product of a symmetric and a skew-symmetric matrix is traceless.
\end{proof}

%%%%%%%%%%%%%%%%%%%%%%%%%%%%%%%%%%%%%%%%%%%%%%%%%%%%%%%%

\section{Eigenfunctions on $\SO{2n}/\U n$ and $\SU{2n} / \Sp n$}\label{section-SO2n-Un}

The aim of this section is to prove the statements presented in Table \ref{table-eigenfunctions} concerning the compact symmetric space $\SO{2n}/\U n$ and $\SU{2n}/\Sp n$. 

\medskip

Firstly, consider the map $\Phi : \SO{2n} \to \SO{2n}$
\begin{equation*}
\Phi(x) = x\cdot J_n \cdot x^t.
\end{equation*}
A simple calculation shows that this map is $\U n$-invariant.
As before, we now wish to consider the function $\phi : \SO{2n} \to \cn$ given by
\begin{equation*}
\phi(x) = \trace(A\Phi(x)) = \trace(x^tAxJ_n)
\end{equation*}
for some skew-symmetric matrix $A \in \cn^{2n\times 2n}$.
Note that in this case the skew-symmetry condition can be assumed without loss of generality, since $\Phi(x)$ is skew-symmetric.

\begin{proposition}\label{proposition-SO2n/Un}
Let $V \subset \cn^{2n}$ be an isotropic subspace and let $a,b\in V \setminus\{0\}$ be linearly independent.
Let $A \in \cn^{2n\times 2n}$ be the skew symmetric matrix
\begin{equation*}
	A = \sum_{i,j=1}^{2n} a_i b_j Y_{ij}.
\end{equation*}
Define the function $\phi:\SO{2n} \to \cn$ by
\begin{equation*}
\phi(x) = \trace(A\Phi(x)) = \trace(x^tAxJ_n) = -\sum_{j,\alpha=1}^{2n} A_{j\alpha}\Phi_{j\alpha}(x),
\end{equation*}
Then $\phi$ is a $\U n$-invariant eigenfunction on $\SO{2n}$ satisfying
\begin{equation*}
\tau(\phi) = -2(n-1)\cdot \phi
\quad\text{and}\quad
\kappa(\phi,\phi) = -\phi^2.
\end{equation*}
\end{proposition}

\begin{proof}
Using (\ref{eq-coordinate-functions-SOn}), similar calculations as in the proof of Proposition \ref{proposition-SUn/SOn} show that
\begin{eqnarray*}
\tau(\Phi_{j\alpha})&=&-2(n-1)\cdot\Phi_{j\alpha}\\
\kappa(\Phi_{j\alpha},\Phi_{k\beta}) &=& -(\Phi_{j\beta}\Phi_{k\alpha}+\Phi_{jk}\Phi_{\alpha\beta})-(\delta_{k\alpha}\delta_{j\beta}-\delta_{jk}\delta_{\alpha\beta}),
\end{eqnarray*}
where in the proof of the latter formula one also uses the fact that $xx^t = I_{2n}$ for all $x \in \SO{2n}$ in order to simplify the terms containing Kronecker deltas.

The formula for $\tau(\phi)$ is thus immediate by linearity.
For the conformality operator, we first note that
\begin{equation*}
A_{j\alpha} = \frac{1}{\sqrt 2}(a_jb_\alpha - a_\alpha b_j).
\end{equation*}
So we calculate
\begin{eqnarray*}
\kappa(\phi, \phi)
&=&
\sum_{j\alpha k\beta} A_{j\alpha}A_{k\beta} \,\kappa(\Phi_{j\alpha},\Phi_{k\beta})\\
&=&
-\sum_{j\alpha k\beta} A_{j\alpha}A_{k\beta}(\Phi_{j\beta}\Phi_{k\alpha} + \Phi_{jk}\Phi_{\alpha\beta})\\
& &
\qquad - \sum_{j\alpha k\beta} A_{j\alpha}A_{k\beta} (\delta_{k\alpha}\delta_{j\beta} - \delta_{jk}\delta_{\alpha\beta})\\
&=&
-\frac12\sum_{j\alpha k\beta} (a_jb_\alpha - a_\alpha b_j)(a_kb_\beta - a_\beta b_k)(\Phi_{j\beta}\Phi_{k\alpha} + \Phi_{jk}\Phi_{\alpha\beta})\\
& &
\qquad- \sum_{jk} A_{jk}A_{kj} + \sum_{j\beta} A_{j\beta}A_{j\beta}.
\end{eqnarray*}
Now there are two types of terms appearing in the first sum, namely
\begin{equation*}
\left(\sum_{j\alpha} a_jb_\alpha \Phi_{j\alpha}\right)^2 = \frac12 \phi^2,
\end{equation*}
which appears 4 times with a plus sign in front,
and
\begin{equation*}
\left(\sum_{j\alpha} a_ja_\alpha \Phi_{j\alpha}\right)\left(\sum_{k\beta} b_kb_\alpha \Phi_{k\beta}\right)=0,
\end{equation*}
where the latter equality follows since $\Phi_{j\alpha}$ is skew-symmetric (even if this were not the case, these terms would cancel out since they appear twice with a plus sign and twice with a minus sign in front).
Thus, we get
\begin{eqnarray*}
\kappa(\phi,\phi) = -\phi^2 + 2\sum_{jk} A_{jk}A_{jk} = -\phi^2 + 2[(a,a)(b,b) - (a,b)^2] = -\phi^2, 
\end{eqnarray*}
since $a$ and $b$ belong to an isotropic subspace $V$ of $\cn^{2n}$.
\end{proof}

Finally, the construction on the space $\SU{2n} / \Sp n$ works in a similar manner, with even weaker assumptions.
Here, we consider the $\Sp n$-invariant mapping $\Phi : \SU{2n} \to \SU{2n}$ given by
\begin{equation*}
\Phi(z) = z\cdot J_n \cdot z^t.
\end{equation*}

\begin{proposition}
For non-zero linearly independent elements $a,b\in\cn^{2n}$, let $A\in \cn^{2n\times 2n}$ be the skew-symmetric matrix
\begin{equation*}
	A = \sum_{i,j=1}^{2n} a_i b_j Y_{ij}.
\end{equation*}	
Define the function $\phi:\SU{2n} \to \cn$ by
\begin{equation*}
\phi(z) = \trace(A\Phi(z)) = \trace(z^tAzJ_n) = -\sum_{j\alpha} A_{j\alpha}\Phi_{j\alpha}(z).
\end{equation*}
Then $\phi$ is a $\Sp n$-invariant eigenfunction on $\SU{2n}$ satisfying
\begin{equation*}
\tau(\phi) = -\frac{2(2n^2-n-1)}{n}\cdot \phi
\quad\text{and}\quad
\kappa(\phi,\phi) = -\frac{2(n-1)}{n} \cdot \phi^2.
\end{equation*}
\end{proposition}

\begin{proof}
The proof is similar as that of Proposition \ref{proposition-SO2n/Un}, where one instead uses the identities (\ref{eq-coordinate-functions-SUn}).
\end{proof}

\begin{remark}
Note that, unlike in Proposition \ref{proposition-SO2n/Un}, the isotropy condition on $a,b$ is not required in this case.
For the case of $\SO{2n} / \U n$ one needs the isotropy condition  because of the final term in $\kappa$ which appears as a consequence of the Kronecker deltas from formula (\ref{eq-coordinate-functions-SOn}), whereas these deltas do not appear in the formula (\ref{eq-coordinate-functions-SUn}) for the coordinate functions on $\SU{2n}$.
\end{remark}

\section{Acknowledgements}

The authors would like to thank Fran Burstall for useful discussions on this work.

Anna Siffert gratefully acknowledges the
supports of the Deutsche Forschungsgemeinschaft (DFG, German Research Foundation) - Project-ID 427320536 - SFB 1442, as well as Germany's Excellence Strategy EXC 2044 390685587, Mathematics M\"unster: Dynamics-Geometry-Structure.

Marko Sobak gratefully acknowledges the
support of Germany's Excellence Strategy EXC 2044 390685587, Mathematics M\"unster: Dynamics-Geometry-Structure.

\appendix

%%%%%%%%%%%%%%%%%%%%%%%%%%%%%%%%%%%%%%%%%%%%%%%%%%%%%%%%

\section{Coordinate functions on $\Sp n$}\label{appendix-coordinate-spn}

The aim of this appendix is to show the following result.

\begin{lemma}\label{lemma-coordinate-spn}
	Let $q_{j\alpha} : \Sp n \to \cn$ denote the standard coordinate functions on $\Sp n$.
	Then
	\begin{eqnarray*}
		\tau(q_{j\alpha}) &=& -\frac{2n+1}{2} \cdot q_{j\alpha},\\
		\kappa(q_{j\alpha}, q_{k\beta}) &=& -\frac{1}{2} \cdot q_{k\alpha}q_{j\beta} + \frac{1}{2} (J_n)_{jk}(J_n)_{\alpha\beta}.
	\end{eqnarray*}
\end{lemma}

Here, we represent $\Sp n$ as the subgroup of the unitary group $\U{2n}$ consisting of elements of the form
\begin{equation*}
q = \begin{bmatrix}
z & w\\ -\bar w & \bar z
\end{bmatrix}.
\end{equation*}
Note therefore that
\begin{equation*}
q_{j\alpha}
=
\begin{cases}
z_{j\alpha} & \text{if }\, 1\leq j,\alpha \leq n\\
w_{j,\alpha-n} & \text{if }\, 1 \leq j \leq n \,\text{ and }\, n+1 \leq \alpha \leq 2n\\
-\bar{w}_{j-n,\alpha} & \text{if }\, n+1\leq j \leq 2n \,\text{ and }\, 1 \leq \alpha \leq n\\
\bar{z}_{j-n,\alpha-n} & \text{if }\, n+1 \leq j,\alpha \leq 2n.
\end{cases}
\end{equation*}
Thus, if we take $1 \leq j,k \leq n$, we see that $(J_n)_{jk} = 0$, and we recover all identities from Lemma 6.1 in \cite{Gud-Mon-Rat-1}. 
The new insight here is that there is an extra term when one mixes the coordinates and their conjugates in kappa.

For the proof of Lemma \ref{lemma-coordinate-spn}, we will also need the following identities, the proof of which can be found in Appendix A.1 of \cite{Sob-MSc} (also stated without proof in \cite{Gud-Sak-1}):
\begin{eqnarray}\label{eq-identities}
\sum_{1\leq r < s \leq n} X_{rs}E_{\alpha\beta}X_{rs}^t &=& \frac12 \delta_{\alpha\beta} I_n + (-1)^{\delta_{\alpha\beta}}E_{\beta\alpha}, \nonumber\\
\sum_{1\leq r < s \leq n} Y_{rs}E_{\alpha\beta}Y_{rs}^t &=& \frac12 \delta_{\alpha\beta} I_n - \frac12 E_{\beta\alpha}, \\
\sum_{t=1}^n D_tE_{\alpha\beta}D_t^t &=& \delta_{\alpha\beta}E_{\beta\alpha}. \nonumber
\end{eqnarray}

\begin{proof}[Proof of Lemma \ref{lemma-coordinate-spn}]
	The formula for $\tau$ follows from Lemma 6.1 in \cite{Gud-Mon-Rat-1} by complex linearity, so we only focus on the formula for $\kappa$.
	Let
	\begin{eqnarray*}
		\B &=& 
		\left\{
		\tfrac{1}{\sqrt 2}
		\begin{bmatrix}
			Y_{rs} & 0\\
			0 & Y_{rs}
		\end{bmatrix},
		\;
		\tfrac{1}{\sqrt 2}
		\begin{bmatrix}
			iX_{rs} & 0\\
			0 & -iX_{rs}
		\end{bmatrix},
		\;
		\tfrac{1}{\sqrt 2}
		\begin{bmatrix}
			iD_t & 0\\
			0 & -iD_t
		\end{bmatrix},
		\right.
		\\[0.1cm]
		&&
		\tfrac{1}{\sqrt 2}
		\begin{bmatrix}
			0 & X_{rs}\\
			-X_{rs} & 0
		\end{bmatrix},
		\;
		\tfrac{1}{\sqrt 2}
		\begin{bmatrix}
			0 & iX_{rs}\\
			iX_{rs} & 0
		\end{bmatrix},\\
		&&
		\left. 
		\tfrac{1}{\sqrt 2}
		\begin{bmatrix}
			0 & D_t\\
			-D_t & 0
		\end{bmatrix},
		\;
		\tfrac{1}{\sqrt 2}
		\begin{bmatrix}
			0 & iD_t\\
			iD_t & 0
		\end{bmatrix} \qquad\mid\;
		\begin{array}{c} 1 \leq r < s \leq n, \\ 1 \leq t \leq n \end{array}
		\right\}
	\end{eqnarray*}
	be the standard orthonormal basis for $\sp n$. Then
	\begin{eqnarray*}
	\kappa(q_{j\alpha},q_{k\beta})
	&=&
	\sum_{Q \in \B} (qQ)_{j\alpha}(qQ)_{k\beta}
	=
	\sum_{Q \in \B} (qQ)_{j\alpha}(Q^tq^t)_{\beta k}\\
	&=&
	\left(q \left\{\sum_{Q\in \B} QE_{\alpha\beta}Q^t \right\} q^t\right)_{jk}.
	\end{eqnarray*}
	Now to calculate the sum we consider four separate cases:
	\begin{enumerate}
		\item[(1)] $1\leq \alpha,\beta \leq n$ so that $E_{\alpha\beta} = \begin{bmatrix}
		E_{\alpha\beta} & 0\\
		0 & 0
		\end{bmatrix}$
		
		\vskip0.2cm
		
		\item[(2)] $1\leq \alpha \leq n$ and $n+1\leq \beta \leq 2n$ so that $E_{\alpha\beta} = \begin{bmatrix}
		0 & E_{\alpha,\beta-n}\\
		0 & 0
		\end{bmatrix}$
		
		\vskip0.2cm
		
		\item[(3)] $n+1\leq \alpha \leq 2n$ and $1\leq \beta \leq n$ so that $E_{\alpha\beta} = \begin{bmatrix}
		0 & 0\\
		E_{\alpha-n,\beta} & 0
		\end{bmatrix}$
		
		\vskip0.2cm
		
		\item[(4)] $1\leq \alpha,\beta \leq n$ so that $E_{\alpha\beta} = \begin{bmatrix}
		0 & 0\\
		0 & E_{\alpha-n,\beta-n}
		\end{bmatrix}$
	\end{enumerate}
	Here, we abuse notation slightly by letting $E_{\alpha\beta}$ denote both its $2n\times 2n$ and its $n \times n$ version, since it is clear from the context which is being used.

	For case (1), we have
	\begin{eqnarray*}
		& &\sum_{Q \in \B} QE_{\alpha\beta}Q^t\\
		&=&
		\begin{bmatrix}
			\tfrac12\sum_{r<s} (Y_{rs}E_{\alpha\beta}Y_{rs}^t - X_{rs}E_{\alpha\beta}X_{rs}^t) - \tfrac12 \sum_t D_t  & 0\\
			0 & 0
		\end{bmatrix}\\
		&+&
		\begin{bmatrix}
			0 & 0 \\
			0 & \tfrac12 \sum_{r<s} (X_{rs}E_{\alpha\beta}X_{rs}^t - X_{rs}E_{\alpha\beta}X_{rs}^t) + \tfrac12 \sum_t (D_tE_{\alpha\beta}D_t^t - D_tE_{\alpha\beta}D_t^t)
		\end{bmatrix}\\[0.2cm]
		&=&
		\begin{bmatrix}
			-\frac12 E_{\beta\alpha} & 0\\
			0 & 0
		\end{bmatrix}
		=
		-\frac12 E_{\beta\alpha}.
	\end{eqnarray*}
	It follows that, in case (1), we have
	\begin{equation*}
	\kappa(q_{j\alpha},q_{k\beta}) = -\frac{1}{2}\cdot q_{k\alpha}q_{j\beta},
	\end{equation*}
	which matches the claimed formula, since in this case $(J_n)_{\alpha\beta}=0$.
	
	For case (2), we get
	\begin{eqnarray*}
		&&\sum_{Q \in \B} QE_{\alpha\beta}Q^t\\
		&=&
		\begin{bmatrix}
			0 & \tfrac{1}{2}\sum_{r<s} (Y_{rs}E_{\alpha,\beta-n}Y_{rs}^t + X_{rs}E_{\alpha,\beta-n}X_{rs}^t) + \sum_t D_t E_{\alpha,\beta-n}D_t^t\\
			0 & 0
		\end{bmatrix}\\
		&+&
		\begin{bmatrix}
			0 & 0\\
			\tfrac12 \sum_{r<s} (-2X_{rs}E_{\alpha,\beta-n}X_{rs}^t) + \tfrac12 \sum_t (-2D_tE_{\alpha,\beta-n}D_t^t) & 0
		\end{bmatrix}.
	\end{eqnarray*}
	Now by the identities (\ref{eq-identities}),
	\begin{eqnarray*}
		&& \sum_{r<s} (Y_{rs}E_{\alpha,\beta-n}Y_{rs}^t + X_{rs}E_{\alpha,\beta-n}X_{rs}^t) + \sum_t D_tE_{\alpha,\beta-n}D_t^t\\
		&=&
		\delta_{\alpha,\beta-n} I_n + \left( \frac{(-1)^{\delta_{\alpha,\beta-n}}}{2} -  \frac12  + \delta_{\alpha,\beta-n} \right)E_{\beta,\alpha-n}\\
		&=&
		\delta_{\alpha,\beta-n} I_n
	\end{eqnarray*}
	since the quantity in the parentheses is 0 regardless of whether $\alpha = \beta-n$ or $\alpha\not= \beta-n$.
	Furthermore,
	\begin{eqnarray*}
		&& \sum_{r<s}X_{rs}E_{\alpha,\beta-n}X_{rs}^t + \sum_t D_tE_{\alpha,\beta-n}D_t^t\\
		&=&
		\frac{1}{2}\delta_{\alpha,\beta-n} I_n + \left( \frac{(-1)^{\delta_{\alpha,\beta-n}}}{2}  + \delta_{\alpha,\beta-n} \right) E_{\beta-n,\alpha}\\
		&=& \frac{1}{2}\delta_{\alpha,\beta-n} I_n + \frac{1}{2}E_{\beta-n,\alpha}
	\end{eqnarray*}
	since, as above, the quantity within the parentheses is always $1/2$.
	Thus
	\begin{eqnarray*}
		\sum_{Q \in \B} QE_{\alpha\beta}Q^t
		&=&
		\begin{bmatrix}
			0 & \tfrac{1}{2} \delta_{\alpha,\beta-n}I_n\\
			-\tfrac{1}{2} \delta_{\alpha,\beta-n}I_n - \tfrac12 E_{\beta-n,\alpha} & 0
		\end{bmatrix}\\
		&=&
		- \frac{1}{2}E_{\beta\alpha}
		+ \frac12 \delta_{\alpha,\beta-n}J_n.
	\end{eqnarray*}
	Now since $qJ_n q^t = J_n$ for each $q \in \Sp n$ we get
	\begin{equation*}
	\kappa(q_{j\alpha}, q_{k\beta}) = -\frac12 \cdot q_{k\alpha}q_{j\beta} + \frac12 \cdot \delta_{\alpha,\beta-n} \cdot (J_n)_{jk},
	\end{equation*}
	which matches the claimed formula (note that $\delta_{\alpha,\beta-n} = (J_n)_{\alpha\beta}$ in this case).
	
	In case (3) one can show using similar calculations as in case (2) that
	\begin{equation*}
	\kappa(q_{j\alpha}, q_{k\beta}) = -\frac12 \cdot q_{k\alpha}q_{j\beta} - \frac12 \cdot \delta_{\alpha-n,\beta} \cdot (J_n)_{jk}.
	\end{equation*}
	Case (4) can be treated in a similar way as case (1) to get
	\begin{equation*}
	\kappa(q_{j\alpha}, q_{k\beta}) = -\frac12 \cdot q_{k\alpha}q_{j\beta}.
	\end{equation*}
\end{proof}

%%%%%%%%%%%%%%%%%%%%%%%%%%%%%%%%%%%%%%%%%%%%%%%%%%%%%%%%

\end{document}